\documentclass[10pt, a4paper]{article}
\usepackage[dvips]{graphicx}
\usepackage{amsmath,amssymb,amscd}
\usepackage{verbatim}
\usepackage{amsthm}
\usepackage{mathtools}
\usepackage{comment}
\usepackage[all]{xy}
\newcommand{\tind}{\mathrm {t\mathchar`-ind}}
\newcommand{\aind}{\mathrm {a\mathchar`-ind}}
\newcommand{\ind}{\mathrm {ind}}
\newcommand{\id}{\mathrm{Id}}

\newcommand{\pz}{\partial_0 }
\newcommand{\po}{\partial_1 }

\newcommand{\C}{\mathbb{C}}
\newcommand{\Z}{\mathbb{Z}}
\newcommand{\R}{\mathbb{R}}
\newcommand{\Rp}{\mathbb{R} _ {\geq 0}}

\newcommand{\Hom}{\mathrm{Hom} \,}

\newcommand{\interior}{\mathrm {int} \,}
\newcommand{\pbu}{{}^{\Phi , b}}
\newcommand{\pbs}{_{\Phi,b}}
\newcommand{\pseu}[1]{\Psi^{#1}_{\Phi,b}(X;E,F)}
\newcommand{\sus}[1]{\Psi^{#1}_{\mathrm{sus}({}^{\Phi , b} N Y)} (\partial_0 X;E,F)}
\newcommand{\inv}[1]{\Psi^{#1}_{\Phi,b, \mathrm{inv}} (\widetilde{\po X} ;E,F)}
\newcommand{\scat} {\Psi^0_{sc,b,\Z/k}(X;E,F)}

\theoremstyle{proposition}
\newtheorem{definition}{Definition}
\newtheorem{proposition}{Proposition}
\newtheorem{lemma}{Lemma}

\newtheorem{theorem}{Theorem}

\title{Fibered Cusp b-Pseudodifferential Operators and Its Applications}
\author{Jun Watanabe}
\begin{document}
\maketitle

\begin{abstract}
Let $X$ be a smooth compact manifold with corners which has two embedded boundary hypersurfaces $\pz X , \po X$,
and a fiber bundle $\phi:\pz X \to Y$ is given. By using the method of blowing up, we define a pseudodifferential culculus $\Psi ^* \pbs (X)$ generalizing the $\Phi$-calculus of Mazzeo and Melrose and  the (small) $b$-calculus of Melrose.
We discuss the Fredholm condition of such operators 
and prove the relative index theorem.
And as its application, the index theorem of ``non-closed'' $\Z/k$ - manifolds is proved.

\end{abstract}

\tableofcontents

\section{Introduction}
To investigate the index problems on a singular space, it is important to define a suitable pseudodifferential calculus adapted to singularities.
Let $X$ be a smooth compact manifold with corners which has two embedded boundary hypersurfaces $\pz X , \po X$,
and a fiber bundle $\phi:\pz X \to Y$ is given. For such $X$, we define a pseudodifferential culculus $\Psi ^* \pbs (X)$ generalizing the $\Phi$-calculus of Mazzeo and Melrose \cite{MM} or  the $b$-calculus of Melrose \cite{M_b}, when $\po X$ or $\pz X$ is empty (respectively).
 We call an element of  $\Psi ^* \pbs (X)$ a fibered cusp $b$-pseudodifferential operator. The purpose of this paper is to give a relative index formula for fibered cusp $b$-pseudodifferential operators,
and as its application, to prove the index theorem of ``non-closed'' $\Z/k$ - manifolds.
For simplicity, we use the 0-th order operators $\Psi^0 \pbs(X)$ for the most argument.
Before considering our general pseudodifferential culculus $\Psi ^0 \pbs (X)$, let us review the $b$-calculus and the $\Phi$-calculus.

First we review the $b$-calculus. (\cite{M_b},\cite{MN})
Let $X$ be a compact manifold with boundary, and $x$ be its boundary defining function. Then we can define a small calculus of $b$-pseudodifferential operators.
 Each element $P \in \Psi^0_b(X)$ defines a bounded operator.
 $$
 P : L^2_b(X) \to L^2_b(X)
 $$
  There are two important homomorphisms:the symbol map $\sigma$ and the normal map $N$ (or the indicial map). These maps are $*$-homomorphisms of filtered algebras which make the following sequences exact.
$$
0 \to \Psi^{-1}_b(X) \to \Psi^{0}_b(X) \xrightarrow{\sigma} S^{0}({}^bT^*X) \to 0
$$

$$
0 \to x \Psi^{-1}_b(X) \to \Psi^{0}_b(X) \xrightarrow{N} \Psi^0_{b,\mathrm{inv}} (\widetilde{\partial X}) \to 0
$$
Where $^bT^*X$ is a vector bundle over $X$ which is non-canonically isomorphic to $T^*X$, and $S^{0}(^bT^*X)$ is
a space of symbols of order $0$ over $^bT^*X$. $\widetilde{\partial X}$ is a
compactification of the positive normal bundle of $\partial X \hookrightarrow X$ which is non-canonically diffeomorphic to $\partial X \times [0,1]$, and $\Psi^m_{b,\mathrm{inv}} (\tilde{X})$ is an algebra of $b$-pseudodifferential operators on $\tilde{X}$
which are invariant under the action of $(0,\infty)$ on $\tilde{X}$.

We say $P \in \Psi^0_b(X)$ is elliptic if $\sigma(P)$ is invertible, and fully-elliptic if in addition
$\hat{N}(P)(\lambda)$ is invertible for all $\lambda \in \R$. Where $\hat{N}(P)$ is a Mellin transform of $N(P)$,
which is a $\Psi^0(\partial X)$-valued entire holomorphic function. It is known that   $P : L^2_b(X) \to L^2_b(X)$ is Fredholm if and only if $P$ is fully-elliptic.

The relative index theorem in \cite{M_b} 
combined with the operator-valued logarithmic residue theorem \cite{GS} gives the following result for a elliptic $P$.
\begin{equation}
\ind(x^{\beta_1}P x^{-\beta_1})- \ind (x^{\beta_2}P x^{-\beta_2}) = \frac{1}{2 \pi i} \mathrm{tr} \oint
 \hat{N}(P)^{-1} (\lambda) \frac{\partial  \hat{N}(P)} { \partial \lambda} (\lambda) d \lambda, \label{rit}
\end{equation}

where $\mathrm{tr}$ is the trace, $\beta_1, \beta_2 \notin - \mathrm{ImSpec} (\hat{N}(P)) \ (i=1,2)$ , $\beta_2>\beta_1$, and the path of
integral is chosen so that its interior contains all poles of $\hat{N}(P)^{-1}(\lambda)$
such that $\beta_1 < -\mathrm{Im}(\lambda) < \beta_2$. And $\mathrm{Spec} (\hat{N}(P)) := \{ \lambda \in \C \mid \text{ $\hat{N}(P)(\lambda)$ is not invertible } \} $ is a discrete set.

Recall that a ``closed'' $\Z/k$-manifold $X$ is a manifold with boundary such that 
$\partial X$ is a disjoint union of $k$ copies of a closed manifold $Y$ , $\partial X = k Y$. Freed and Melrose \cite{FM} introduced a subalgebra of 
$\Psi^0_b(X)$, sconsists of $P \in \Psi^0_b(X)$ for which $N(P)$ can be written as a direct sum of $k$ copies of some operator. For such operator $P$, $\ind (x^\beta P x^{-\beta}) \mod k \in \Z/k $ is independent of $\beta$ because right hand side of the formula (\ref{rit}) is always a multiple of $k$. And they proved the index theorem which asserts that this $\Z/k$-valued index can be
written in terms of topological K theory.

Secondly we review the $\Phi$-calculus \cite{MM}. Let $X$ be a compact manifold with boundary,
and a fiber bundle $\phi:\partial X \to Y$ is given. Such $X$ is called a manifold with fibered boundary.
Fix a boundary defining function $x$ of $\partial X$.
Then we can define a calculus of $\Phi$-pseudodifferential operators (or fibered cusp pseudodifferential operatos)
$\Psi^*_\Phi(X)$, which is a filtered $*$-algebra.
Each element $P \in \Psi^0_\phi(X)$ defines a bounded operator.
$$
P: L^2_{\Phi}(X) \to L^2_{\Phi}(X)
$$
There exists two homomorphisms, a symbol map $\sigma$ and a normal map $N$ which make the following sequences exact.
$$
0 \to \Psi^{-1}_\Phi(X) \to \Psi^{0}_\Phi(X) \xrightarrow{ \sigma} S^{0}({}^\Phi T^*X) \to 0
$$
$$
0 \to x \Psi^{-1}_\Phi(X) \to \Psi^{0}_\Phi(X) \xrightarrow{N} \Psi^0_{\mathrm{sus}(^\Phi N Y)} (\partial X) \to 0
$$
Where ${}^\Phi T^*X$ is a vector bundle over $X$ which is non-canonically isomorphic to $T^*X$, 
$^\Phi N Y$ is a vector bundle over $Y$ which is non-canonically isomorphic to $\R \oplus T^*Y$
and $\Psi^0_{\mathrm{sus}(^\Phi N Y)} (\partial X)$ is a space of $^\Phi N Y$-suspended pseudodifferential operators on
$\partial X$ of order 0.

We say $P \in \Psi^0_\phi(X)$ is elliptic if $\sigma(P)$ is invertible, and fully elliptic in addition $N(P)$ is invertible.
It is shown that  $P: L^2_{\Phi}(X) \to L^2_{\Phi}(X)$ is Fredholm
 if and only if $P$ is fully-elliptic.
 
In an extreme case when $\phi$ is an identity map $\phi:\partial X \to \partial X $, 
the calculus $\Psi^0_{sc}(X):= \Psi^0_\Phi(X)$ is called scattering calculus. In this case, as outlined in \cite{M_g}, the index problem of fully elliptic operator is reduced to the Atiyah-Singer index theorem. Let us explain it briefly.
We can define a map 
$$
\{P \in \Psi^0_{sc}(X) \mid \text{$P$ is fully elliptic}\} / \text{homotopy} \to K(D(TX),\partial D(TX)),
$$
where $\partial D(TX) = D(TX|_{\partial X}) \cup S(TX)$ and $D$ or $S$ means a disk or a sphere bundle of the vector space.
The composition of this map and the topological index map \cite{AS}  $\tind:  K(D(TX),\partial D(TX)) \to \Z$
gives the index of fully elliptic operator.

Finally let us move on to our general calculus.
Let $X$ be a smooth compact manifold with corners which has two embedded boundary hypersurfaces $\pz X , \po X$ ,
and a fiber bundle $\phi:\pz X \to Y$ is given. Suppose that the fiber $Z$ of $\phi$ is a closed manifold. Fix a boundary defining function $x_0$ of $\pz X$ and $x_1$ of $\po X$.
Extending the notion in \cite{FM} we call such $X$ also a manifold with fibered boundary.

We define a pseudodifferential calculus $\Psi ^* \pbs (X)$ of fibered cusp $b$-pseudodifferential operators. 
Each element $P \in \Psi^0 \pbs(X)$ defines a bounded operator.
$$
P: L^2 \pbs (X) \to L^2 \pbs (X)
$$

We define a symbol map $\sigma$ and two normal maps
$N_0,N_1$ with respect to two boundaries $\pz X, \po X$, which make the following sequences are exact.
$$
0 \to \Psi^{-1} \pbs (X) \to \Psi^{0} \pbs (X) \xrightarrow{\sigma} 
S^{0}(\pbu T^*X) \to 0
$$
$$
0 \to x_0 \Psi^{0} \pbs (X) \to \Psi^{0} \pbs (X) \xrightarrow{N_0} \Psi^0_{\mathrm{sus}(\pbu N Y)} (\pz X) \to 0
$$
$$
0 \to x_1 \Psi^{0} \pbs(X) \to \Psi^{0}_b(X) \xrightarrow{N_1} \Psi^0_{\Phi,\mathrm{inv}} (\widetilde{\po X}) \to 0
$$
Where $\pbu T^*X$ is a vector bundle over $X$ which is non-canonically isomorphic to $T^* X$
, $\pbu N Y$ is a vector bundle over $Y$ which is non-canonically isomorphic to $\R \oplus ^bT^*Y$. $\widetilde{\po X}$ is a compactification of the normal bundle of the embedding 
$\po X \hookrightarrow X$ which is non-canonically diffeomorphic to $\po X \times [0,1]$.
Because $\widetilde{\po X}$ is also a manifold with fibered boundary so we can define 
$\Psi^0_{\Phi} (\widetilde{\po X})$, and ``$\mathrm{inv}$'' in $\Psi^0_{\Phi,\mathrm{inv}} (\widetilde{\po X})$ means the invariance under the action
of $(0,\infty)$.

We say $P \in \Psi ^m \pbs (X)$ is elliptic when  $\sigma(P)$ is invertible, and fully-elliptic when in addition $N_0(P)$ and $\hat{N_1}(P)(\lambda) \ (\lambda \in \R)$ are invertible.
Where $\hat{N_1}(P)$ is a Mellin transform of $N_1(P)$, which is a $\Psi^0_\Phi (\po X)$-valued entire holomorphic function.

And we prove the Fredholm condition.
\begin{theorem}
For $P: L^2 \pbs (X) \to L^2 \pbs (X)$ is Fredholm if and only if $P$ is fully elliptic.
\end{theorem}

We also prove the relative index theorem.
\begin{theorem}
Let $P \in \Psi ^0 \pbs (X)$ and suppose $\sigma(P)$ and $N_0(P)$ are invertible.
Take any $\beta_i \notin - \mathrm{ImSpec}(\hat{N_1}(P)) \ (i=1,2)$, $\beta_2>\beta_1$. Then,
\begin{equation}
\ind (x_1^{\beta_1} P x_1 ^{-\beta_1})- \ind (x_1^{\beta_2} P x_1 ^{-\beta_2}) = \frac{1}{2 \pi i} \mathrm{tr} \oint
 \hat{N_1}(P)^{-1} (\lambda) \frac{\partial  \hat{N_1}(P)} { \partial \lambda} (\lambda) d \lambda, \label{rit2}
\end{equation}
where the path of
integral is chosen so that its interior contains all poles of $\hat{N_1}(P)^{-1}(\lambda)$ such that $\beta_1 < -\mathrm{Im}(\lambda) < \beta_2$.
\end{theorem}

Let $X$ be a $\Z/k$-manifold, i.e. $X$ is a smooth compact manifold with corners which has two embedded boundary hypersurfaces $\pz X , \po X$, 
and $\po X = k Y$ for some manifold with boundary $Y$. Then $X$ is a manifold with Baas-Sullivan singularity \cite{Baas},
and $X$ is called closed as a manifold with Baas-Sullivan singularity if $\pz X$ is empty.
We regard $X$ as a manifold with fibered boundary by setting $\phi=\id_{\pz X}:\pz X \to \pz X$.
And define $\Phi^0_{sc,b}:= \Phi^0 \pbs$. We define a subalgebra $\Psi^0 _{sc,b,\Z/k}(X)$ of $\Psi^0 _{sc,b} (X)$ which is compatible to
the structure as a $\Z/k$-manifold, by setting 
$$\Psi^0 _{sc,b,\Z/k}(X) = \{P \in \Psi^0 \pbs (X) \mid \text{ $N_1(P)$  can be written as a direct sum of $k$ copies of some operator 
on $Z$ }\} .$$

For $P \in \Psi^0 _{sc,b,\Z/k}(X) $, $\ind_\beta(P) \mod k \in \Z/k$ is independent of $\beta$  because the right hand side of (\ref{rit2}) is always a multiple of $k$.

 We can define a map $$
\{P \in \Psi^* _{sc,b,\Z/k}(X) \mid \text{ $\sigma (P)$ and $N_0(P)$ are invertible} \} / \text{homotopy}
\to K (D(\overline{TX}), \pz D (\overline{TX})),
$$
where the overlines mean the identification of $k$ copies, and $\overline{TX}\to \overline{X}$ is a vector bundle. 
$\partial D (TX) = S (TX) \cup D(TX|_{\pz X}) 
\cup D( TX|_{\po X})$,  and $\pz D(\overline{TX}) := S (\overline{TX}) \cup D(\overline{TX|_{\pz X}})$. As in the case of Atiyah-Singer \cite{AS} 
or Freed-Melrose \cite{FM}, there exists a topological index map
$\tind: K (D(\overline{TX}), \pz D (\overline{TX})) \to \Z/k$ \cite{W}.
And we prove the composition of these two maps gives the $\Z/k$-valued index of the operator $P$.

\section{The definition of fibered cusp b-pseudodifferential operators}\label{def}
In this section, based on the discussion in \cite{MM}, we define the fibered cusp $b$-pseudodifferential operators.

Let $X$ be a smooth compact manifold with corners which has two embedded boundary hypersurfaces. 
Thus, following relations hold where $\partial_0 X$ and $\partial_1 X$ are the boundary hypersurfaces.

$$\partial X = \partial _0 X \cup \partial _1 X \  , \  \angle X = \partial _0 X \cap \partial _1 X$$

Suppose a fiber bundle $\phi : \partial _0 X \to Y$ is given, where $Y$ is a compact manifold
with boundary and each fiber $Z$ is a compact manifold without boundary.
Suppose further $\phi$ maps the boundary to the boundary, thus restricts to a fiber bundle
$\phi | _{\angle X} : \angle X \to \partial Y$, and the following diagram commutes.

$$
\xymatrix{
 \angle X  \ar@{^{(}->}[r] \ar[d]^{\phi_{\angle X}} & \pz X  \ar[d]^{\phi} \\
 \partial Y \ar@{^{(}->}[r]& Y
}
$$

We say $X$ is a manifold with fibered boundary in this case.

The whole following discussion can be applied in the case when each fiber $\phi ^{-1} (y)$ varies on
 connected components of $Y$, but for simplicity, we assume they are all diffeomorphic to a
 single closed manifold $Z$.

Take any boundary defining functions $x_0 , x_1 \in C^\infty (X) $ for $\partial_0 X , \partial_1 X$ respectively. We assume that $x_1|_{\pz X}$ is constant on each fiber of $\phi$, which is always possible by taking boundary defining function of $\partial Y \subset Y$ and pulling it back to $\pz X$
and extend it to $X$.

To describe local properties of $X$, it is convenient to introduce ``model space'' $M$ as follows.
$$
M:= \{(x_0,x_1 ,y,z) \in \Rp \times \Rp \times \R ^{k-1} \times \R^{l} \} , \ 
\pz M:=\{x_0 =0 \} , \ 
\po M:=\{x_1=0 \},
$$

$$
N:=\{(x_1 , y) \in \Rp \times \R^{k-1} \} , \  \phi:\po M \to N
$$
where $\phi$ is the projection. Then $M$ is a manifold with fibered boundary (without a compactness assumption). 

For any manifold with fibered boundary $X$ and $p \in X$, their exists an open neighbourhood $p \in U$ and diffeomorphism onto open subset of $M$ which preserves the structure of manifold with fibered boundary. Where ``preserving the structure'' means it preserves $\partial_0$,
$\partial_1$ and $\phi$, and when $p \in \pz X$ or $p \in \po X$, it preserves the function $x_0$ or $x_1$.

We define fibered cusp $b$-vector fields on $X$ as follows:
$$\mathcal{V}_{\Phi , b} (X) = 
\{V \in \mathcal{V}(X) \mid Vx_0 \in x_0^2 C^\infty (X) , V|_{\pz X} \text { is tangent to the fibers of  $\phi$, $V|_{\po X}$  is tangent to   $\po X$ }  \},  $$
where $\mathcal{V}(X)$ is the space of smooth vector fields on $X$.

When $X=M$, it is straightforward to check $\mathcal{V}_{\Phi , b} (M)$ is freely generated by 
$x_0^2 \cfrac{\partial}{\partial x_0} , x_0 x_1 \cfrac{\partial}{\partial x_1} ,
   x_0 \cfrac{\partial}{\partial y_i} , \cfrac{\partial}{\partial z_j} $ over $C^ \infty (M)$.
Thus their exist a smooth vector bundle $^{\Phi,b} TX$ over $X$ and the isomorphism $\Gamma (X, ^{\Phi,b} TX) \simeq \mathcal{V}_{\Phi , b} (X)$.

 The map $\pbu TX \to TX$, induced by $\mathcal{V}\pbu(X) \hookrightarrow \mathcal{V}(X)$ defines the Lie algebroid structure on $X$.
Such a structure is called a Lie structure at infinity in \cite{ALN1} and \cite{ALN2}.

Although the space $\mathcal{V}_{\Phi , b} (X)$ depends on the choice of boundary defining function $x_0$ of $\pz X$, the full information of $x_0$ is not needed to determine $\mathcal{V}_{\Phi , b} (X)$, and we can prove the following lemma by direct calculation.
\begin{lemma} \label{boundary}
Two choices of boundary defining function of $\pz X$, $x_0$ and $\tilde{x}_0$ defines a same space $\mathcal{V}_{\Phi , b} (X)$ if and only if $\tilde{x}_0 / x_0 = \alpha \in C^\infty (X)$ satisfies $\alpha |_{\pz X} = \phi^* \gamma$ for some $\gamma \in C^\infty(Y)$.
\end{lemma}

On the set 
$$B:= \{ x_0 \mid \text{$x_0$ is a boundary defining function of $\pz X$} \} ,$$
the group 
$$G:= \{ \alpha \in C^{\infty}(X) \mid \text{$\alpha>0$  and $\alpha|_{\pz X}$ is constant on each fiber of $\phi$ } \}$$
 acts by multiplication.
The above lemma implies that fixing $\mathcal{V}\pbs (X)$ is equivalent to fixing the $G$ orbit in $B$.

From now on, we fix the Lie algebroid $\mathcal{V}\pbs (X)$, or equivalently, the $G$ orbit in $B$.

Next we define the vector bundle over $Y$ by using the local coordinate:
$$
\pbu NY := \mathrm{span} \{ x_0^2 \frac{\partial}{\partial x_0}, x_0 x_1 \frac{\partial}{\partial x_1} , 
x_0 \frac{\partial}{\partial y_i} \}.
$$

By definition,
$$
\phi^*(\pbu NY) =\ker (\pbu TX|_{\pz X} \to TX|_{\pz X} ) .
$$

We want to describe this vector bundle without using coordinate.
Note that if we fix $x_0$, $\pbu NY$ is clearly isomorphic to $\underline{\R} \oplus {}^b TY$.
If $\tilde{x}_0 = \alpha x_0$ is another choice of boundary defining function where $\alpha \in G$.
 The coordinate exchange is given as following.
$$
 x_0^2 \frac{\partial}{\partial x_0} = \frac{1}{\alpha} \cdot \tilde{x}_0^2 \frac{\partial}{\partial \tilde{x}_0} 
$$

$$
x_0 x_1 \frac{\partial}{\partial x_1} =  \frac{1}{\alpha^2}  \frac{\partial \alpha}{\partial x_1} 
\cdot \tilde{x}_0^2  x_1 \frac{\partial}{\partial \tilde{x}_0} + \frac{1}{\alpha} \cdot \tilde{x}_0 x_1 \frac{\partial}{\partial x_1}
$$

$$
x_0 \frac{\partial}{\partial y_i} = \frac{1}{\alpha^2} \frac{\partial \alpha}{\partial y_i} \cdot  \tilde{x}_0^2 \frac{\partial}{\partial \tilde{x}_0}
+ \frac{1}{\alpha} \cdot \tilde{x}_0 \frac{\partial}{\partial y_i}
$$

The group $G$ acts on $\underline{\R} \oplus {}^b TY$ by
$$
(\alpha , \tau , \eta) \in G \times (\underline{\R} \oplus {}^b TY) \mapsto 
(\tau/ \gamma + d \gamma \cdot \eta/ \gamma^2 , \eta/\gamma) \in \underline{\R} \oplus {}^b TY ,
$$
where $\gamma \in C^\infty(Y)$ is defined by $\phi^*\gamma = \alpha^{-1}|_{\pz X}$ and $d \gamma \cdot \eta$
means a paring of $T^*Y$ and ${}^bTY$. 
By the above formula of coordinate exchange, the map 
$$
(\tau, \sigma_1 x_1 \frac{\partial}{\partial x_1} , \eta_i \frac{\partial}{\partial y_i} , x_0) \in  (\underline{\R} \oplus {}^b TY) \underset{G}\times B
 \mapsto
 (\tau x_0^2 \frac{\partial}{\partial x_0}, \sigma_1 x_0x_1 \frac{\partial}{\partial x_1}, x_0 \frac{\partial}{\partial y_i}) \in \pbu NY  ,
$$
is a well-defined isomorphism. This gives a coordinate-free definition of $\pbu NY$.

Let $X^2_b$ be a smooth manifold obtained by blowing up $\pz X \times \pz X$ and $\po X \times \po X$ in $X \times X$. 
$$
X^2_b = [X^2 ; (\pz X)^2 , (\po X)^2 ]\  , \beta_b:X^2_b \to X^2 
$$
The order of two blow-ups does not matter because $\pz X \times \pz X$ and $\po X \times \po X$ intersects transversely (see \cite{M_c}).
$X^2_b$ has 6 boundary hypersurfaces $L_0,F_0,R_0,L_1,F_1$ and $R_1$, 
which corresponds to $\pz X \times X , \pz X \times \pz X , X \times \pz X , \po X \times X , \po X \times \po X$ , and $X \times \po X$ respectively, where $L$, $F$ or $R$ stands for left, front or right.

 Define $\Phi:= \pz X  \underset{Y} \times \pz X = \{(w,w') \in \pz X \times \pz X \mid \phi(w) = \phi(w') \} \subset (\pz X)^2$.
 Then $\Phi$ can be lifted to $\Phi_b \subset (\pz X)^2_b$.
 The smooth function $x_0'/x_0: X^2_b \to [0,\infty]$ is independent of the choice of $x_0$ when restricted to $F_0$ and
there is a diffeomorphism $(\pz X)^2_b \simeq \{x_0'/x_0 =1 \} \cap F_0$.
 
By regarding $\Phi_b$ as a submanifold of $\{x_0'/x_0 =1 \} \cap F_0$, we define
$$
X^2_{\Phi,b} := [X^2_b ; \Phi_b] \ , \beta_\phi : X^2_{\Phi,b} \to X^2_b \ , \beta:= \beta_b \circ \beta_\phi : X^2_{\Phi,b} \to X^2.
$$

$X^2_{\Phi,b}$ has 7 boundary hypersurfaces $L_0,F_0,R_0,L_1,F_1 , R_1$ and $FF_0$
where the new hypersurface $FF_0$ corresponds to $\Phi_b$.

For a model space $M$, we want to describe these blowing-ups explicitly using coordinates.
$$
M^2 = \{ (x_0,x_1,y,z,x_0',x_1',y',z') \mid x_0 \geq 0, x_1 \geq 0, x_0' \geq 0, x_1' \geq 0 \}
$$

First the coordinate on $M^2_b \setminus L_0 \setminus L_1$ is given as following.
$$
\{ (x_0,x_1,y,z,s_0,s_1,y',z') \mid x_0 \geq 0, x_1 \geq 0, s_0 \geq 0, s_1 \geq 0 \}
$$
$$
s_0=\frac{x_0'}{x_0} , s_1 = \frac{x_1'}{x_1}
$$

In this coordinate, $\Phi_b = \{ x_0=0 , s_0=1,s_1=1,y=y' \}$. Thus, we can give an explicit coordinate on 
$M^2_{\Phi,b} \setminus L_0 \setminus R_0$ around $FF_0$ as following.

$$
\{(x_0,x_1,y,z,u_0,u_1,v,w) \mid x_0 \geq 0, x_1 \geq 0, x_0 \cdot \sqrt{1+u_0^2+u_1^2+v^2}<1\}
$$
$$
u_0= \frac{1-s_0}{x_0} = \frac{x_0-x_0'}{x_0^2} ,\ u_1 = \frac{1-s_1}{x_0} = \frac{x_1-x_1'}{x_0x_1} ,\ v = \frac{y-y'}{x_0},\ w = z-z'
$$

\begin{proposition}\label{groupoid}
$\mathcal{G}:= X^2 \pbs \setminus L_0 \setminus R_0 \setminus L_1 \setminus R_1$ has a structure of a Lie groupoid by extending 
the Lie groupoid structure on $\interior X^2$. 
The set of units of $\mathcal{G}$ is the lifted diagonal $\Delta \pbs \subset X^2\pbs$, and the associated Lie algebroid
 $A (\mathcal{G})$ is $\pbu TX$.
\end{proposition}

\begin{proof}
Recall that the Lie algebroid structure of $\interior X^2$ is given as follows:
$$
d(x_0,x_1,y,z,x_0',x_1',y',z') = (x_0',x_1',y',z') , \ r(x_0,x_1,y,z,x_0',x_1',y',z') =(x_0,x_1,y,z),
$$
$$
\mu((x_0,x_1,y,z,x_0',x_1',y',z') , (x_0',x_1',y',z',x_0'',x_1'',y'',z'')) = (x_0,x_1,y,z,x_0'',x_1'',y'',z'')
$$
$$
u(x_0,x_1,y,z) = (x_0,x_1,y,z,x_0,x_1,y,z) , \  \iota(x_0,x_1,y,z,x_0',x_1',y',z') = (x_0',x_1',y',z',x_0,x_1,y,z) 
$$
where $d,r,\mu,u$ and $\iota$ are domain, range, multiplication, unit, and inversion map respectively.

By using the coordinate on $X^2 \pbs$ described above, we can compute:

$$x_0'= x_0- x_0^2 u_0,\ x_1' = x_1 - x_0x_1 u_1,\ y'=y-x_0 v,\ z' = z-w .$$
\begin{equation}\label{sum}
\frac{x_0-x_0''}{x_0^2} = u_0+ \frac{x_0'^2}{x_0^2} u_0',\ \frac{x_1-x_1''}{x_0x_1}=u_1+ \frac{x_0'x_1'}{x_0x_1} u_1', \ 
\frac{y-y''}{x_0} = v+ \frac{x_0'}{x_0} v',\ z-z''=w+w' ,
\end{equation}
where
$$
u_0'= \frac{x_0'-x_0''}{x_0'^2} ,\ u_1' =\frac{x_1'-x_1''}{x_0'x_1'} ,\ v' = \frac{y'-y''}{x_0'},\ w' = z'-z''.
$$
Because $x_0'/x_0= 1-x_0u_0$ and $x_1'/x_1=1-x_0u_1$ are smooth on $\mathcal{G}$,
$d,r,\mu,u$ and $\iota$ can be extended to $\mathcal{G}$. These maps satisfy the axiom of the Lie groupoid as it satisfy on the dense
subset $\interior X^2$.

Clearly, the set of units of $\mathcal{G}$ is the lifted diagonal 
$$\Delta \pbs = \{ (x_0,x_1,y,z,u_0,u_1,v,w) \mid u_0=u_1=v=w=0 \}.$$
By definition, $A(\mathcal{G})$ is spanned by the restrictions of 
$\partial/\partial u_0, \partial/ \partial u_1, \partial/ \partial v$ and $\partial/ \partial w$ to $\Delta \pbs$.
\begin{equation}\label{vector}
\frac{\partial}{\partial u_0} = -x_0^2 \frac{\partial}{\partial x_0'}, \ \frac{\partial}{\partial u_1} =
 -x_0x_1 \frac{\partial}{\partial x_1'},\ \frac{\partial} {\partial v_i}= - x_0 \frac{\partial}{\partial y'},\ 
 \frac{\partial}{\partial w} = - \frac{\partial}{\partial z'}
\end{equation}
Thus, $A(\mathcal{G})= \pbu TX$.
\end{proof}

The Lie groupoid structure of $\mathcal{G}$ can be described simpler if we use the notion of blow-up of a Lie groupoid(\cite{DS}).
Let $\mathcal{G}_0 \rightrightarrows X$ be the groupoid obtained by blowing up $X^2 \rightrightarrows X$ at $Hol(\pz X, \mathcal{F}) \rightrightarrows \pz X$, where $\mathcal{F}$ is the foliation defined by the fibration $\phi: \pz X \to Y$ and $Hol$ is the holonomy groupoid.
This groupoid $\mathcal{G}_0$ is described more precisely in \cite{DLR}.
Let $\mathcal{G}_1 \rightrightarrows X $ be the groupoid obtained by blowing up $X^2 \rightrightarrows X$ at $(\po X)^2 \rightrightarrows \po X$, which is called a puff groupoid
in \cite{Mo}.
Then $\mathcal{G}$ is the fiber product of $\mathcal{G}_0$ and $\mathcal{G}_1$. 

\begin{definition}
Let $E$ and $F$ are vector bundles over $X$, and $m \in \R$ be an arbitrary real number.
Then, the space of fibered cusp b-pseudodifferential operator of order m from $E$  to $F$ is defined as follows.
\begin{equation*}
\begin{split}
\Psi ^m \pbs (X;E,F):= \{ P \in I^m(X^2 \pbs , \Delta \pbs (X) ; \pi_L^* F \otimes \pi_R^*E ' \otimes 
\pi_R^* (\pbu \Omega)) , \\
 \text{$P$ $\equiv$ $0$ on each boundary hypersurface
except for $FF_0$ or $F_1$} \}
\end{split}
\end{equation*}
Where $ \pi_L,\pi_R :X^2 \pbs \to X$ are left and right projection, $E'$ is a dual of $E$, $\pbu \Omega:= |\Lambda^{\dim X}| (\pbu T^*X) $ , $I^m$ is a space of conormal distribution, and $P \equiv 0$ means a vanishing of infinite order. 

\end{definition}

In this paper we only consider classical or one-step polyhomogeneous conormal distribution.

The space of uniformly supported pseudodifferential operator $\Psi^m_c(\mathcal{G};E,F)$ defined in \cite{NWX} and \cite{ALN2}
is, by definition,
$$
\Psi^m_c(\mathcal{G};E,F)= \{ P \in \pseu{m} \mid \text{ $P$ vanishes identically on the neighborhood of $L_0\cup R_0 \cup L_1 \cup R_1$} \}.
$$
By definition, $\Psi^m_c(\mathcal{G};E,F) \subset \Psi ^m \pbs (X;E,F)$.

Let $\dot{C}^\infty(X;E):= x_0^{\infty} x_1^{\infty} C^{\infty}(X;E)$ be a space of smooth section of $E$ which vanishes in infinite order on $\pz X$ and $\po X$.
  By general theory of conormal distributions \cite{M_c}, for $u \in \dot{C}^\infty(X;E)$ 
  we can see that $Pu:= (\pi_L)_* P \pi^*_R u $ defines continuous linear operators $\dot{C}^\infty(X;E)  \to \dot{C}^\infty(X;F) $
  where $(\pi_L)_*$ is a fiber integral.

To give an explicit description of $P$, we assume $P$ is supported in the coordinate patch $\{(x_0,x_1,y,z,u_0,u_1,v,w) \}$.

By the condition $P \equiv 0$ on $L_0 \cup R_0 \cup L_1 \cup R_1$ 
, $P$ decreases rapidly
as $u_0^2+u_1^2+|v|^2 \to \infty$. Thus we can take a fourier transform with respect to $u_0,u_1,v,w$
, and $P$ can be written as following.
\begin{equation}\label{symbol}
P(x_0,x_1,y,z,u_0,u_1,v,w) = \left( \int e^{i\sigma_0 u_0} e^{i\sigma_1 u_1 } e^{i \eta \cdot v}
e^{i \zeta w} p(x_0,x_1,y,z,\sigma_0,\sigma_1,\eta,\zeta) d \sigma_0 d \sigma_1 d \eta d \zeta
\right) |du_0du_1dvdw| 
\end{equation}
Note that $\pi_R^* (\pbu \Omega)$ is generated by
$x_0'^{-k-2}x_1'^{-1} |dx_0'dx_1'dy'dz'|$.
Its restriction to the fiber of $\pi_L$ is  $\frac{x_0^{k+2} x_1}{x_0'^{k+2} x_1'} du_0du_1dvdw$, 
 and the coefficient $\frac{x_0^{k+2} x_1}{x_0'^{k+2} x_1'}$ is a non-zero smooth function, thus we can absorb this
 coefficient in the symbol term.
 
For a function $u(x_0,x_1,y,z)$, the action of $P$ is given by
\begin{equation*}
\begin{split}
Pu(x_0,x_1,y,z)=\int e^{i\sigma_0 u_0} e^{i\sigma_1 u_1 } e^{i \eta \cdot v}
e^{i \zeta w} p(x_0,x_1,y,z,\sigma_0,\sigma_1,\eta,\zeta) u(x_0-x_0^2u_0,x_1-x_0x_1u_1,y-x_0v,z-w)
\\
d \sigma_0 d \sigma_1 d \eta d \zeta
|du_0du_1dvdw|.
\end{split}
\end{equation*}

 For any complex numbers $\alpha$ and $ \beta$, 
  $x_0^\alpha x_1 ^\beta P x_0^{-\alpha} x_1^{-\beta} \in \Psi^m \pbs (X;E,F)$, because 
  $x_0'/x_0= 1-x_0u_0$ , $x_1'/x_1=1-x_0u_1$ and these derivatives are smooth up to $FF_0$ and $F_1$ and at most polynomial order up to other boundary hypersurfaces.
  Thus $P$ also defines an operator $x_0^\alpha x_1 ^\beta {C}^\infty(X;E) \to x_0^\alpha x_1 ^\beta {C}^\infty(X;F)$. In particular,  we can obtain three operators,
  
$$
P|_{\pz X}: {C}^\infty(\pz X;E) \to  {C}^\infty(\pz X;F),
$$

$$
P|_{\po X}: {C}^\infty(\po X;E) \to  {C}^\infty(\po X;F),
$$
and
$$
P|_{\angle X}: C^\infty(\angle X;E) \to C^\infty(\angle X; F),
$$
such that $(Pu)|_{\pz X} = P|_{\pz X} u|_{\pz X}$, $(Pu)|_{\po X} = P|_{\po X} u|_{\po X}$ and
$(Pu)|_{\angle X} = P|_{\angle X} u|_{\angle X}$ for $u \in C^\infty (X;E)$.

\begin{lemma}
$P|_{\pz X} \in \Psi^m_{\mathrm{fiber}}(\pz X;E,F)$ , 
$P|_{\po X} \in   \Psi^m_{\Phi } (\po X;E,F) $ , 
$P|_{\angle X} \in  \Psi^m_{\mathrm{fiber}} (\angle X;E,F)$
where $ \Psi^m_{\mathrm{fiber}}$ is the space of  a family of m-th pseudodifferential operators on each fiber of $\phi$
which depends smoothly on the base points, and $\Psi^m_\Phi$ is the space of m-th fibered cusp pseudodifferential operator.
\end{lemma}

\begin{proof}
By using partition of unity, we can assume $P$ is supported in the coordinate patch.
Let $p$ is a symbol of $P$ as in (\ref{symbol}), the symbols of $P|_{\pz X}$, $P|_{\po X}$ and $P|_{\pz X}$ are
$p(0,x_1,y,z,0,0,0,\zeta)$ , $p(x_0,0,y,z,\sigma_0,0,\eta,\zeta)$ and 
$p(0,0,y,z,0,0,0,\zeta)$ respectively.
\end{proof}

As in \cite{MM}, we can construct a blow-up $X^3 \pbs$ of $X^3$
$$
X^3 \pbs := [X^3_b; \Phi_T;\Phi_{FT}; \Phi_{ST};\Phi_{CT};\Phi_{F};\Phi_{S};\Phi_{C}].
$$
By using this manifold, we can prove the following proposition exactly parallel as in \cite{MM} or \cite{DLR}.
\begin{proposition}
Let $E,F,G$ are vector bundles over X , $m , m' \in \R$ , $P \in \Psi^m \pbs (X;E,F)$ and $Q \in \Psi^{m'} \pbs (X,F,G)$ , then $Q \circ P \in \Psi^{m+m'}\pbs(X;E,G)$.
\end{proposition}

\section{Symbols and normal operators}\label{normal}
In this section, we define the symbol $\sigma$ and normal operators $N_0$ and $N_1$.
Essentially, $N_0(P)$ and $N_1(P)$ are restriction of the kernel of $P$ to $FF_0$ and $F_1$. A similar notion is called a normal operator in \cite{MM}, an indicial operator in \cite{M_b}, and
just a symbol in \cite{DLR}.

As described in \cite{M_c},\cite{MM},\cite{NWX}, we can obtain a symbol homomorphism

$$
 \sigma : \Psi^m \pbs(X;E,F) \to  S^{m} (\pbu T^*X ; \Hom(E,F)).
$$
Where $S^{m} (\pbu T^*X ; \Hom(E,F))$ is a space of bundle homomorphisms $\pbu T^* X \setminus 0 \to \Hom(E,F)$ which are homogeneous of 
degree $m$.
The sequence
$$
 0 \to \Psi^{m-1} \pbs(X;E,F)\to \Psi^m \pbs(X;E,F) \xrightarrow{\sigma}  S^{m} (\pbu T^*X ; \Hom(E,F)) \to 0
$$
is exact.

Next, we consider a normal operator at $\partial_0 X$.

$p \in Y$ , $(\tau,\tilde{\eta}) \in 
(\underline{\R} \oplus T^*Y )_y$. Fix any real valued $f \in C^\infty(Y)$ such that $f(p) = \tau$ , $df(p) = \tilde{\eta}$
, and real valued $\tilde{f} \in C^\infty(X)$ such that $\phi ^* f = \tilde{f}|_{\pz X}$ .
Define
$$(\tau, \tilde{\eta}) \in (\underline{\R} \oplus T^*Y )_y \mapsto \tilde{N_0}(\tau,\tilde{\eta}) := [\exp(-i\tilde{f}/x_0)P \exp(i\tilde{f}/x_0)]|_{\phi^{-1}(p)} \in 
\Psi^m(\phi^{-1}(p);E,F) $$.

 \begin{eqnarray*}
 \exp \left(  i \frac{\tilde{f}(x_0',x_1',y',z')}{x_0'}  - i \frac{\tilde{f}(x_0,x_1,y,z)}{x_0} - \right)
 &=&\exp \left(i \frac{\tilde{f}(x_0-x_0^2X_0,x_1-x_0x_1X_1,y-x_0Y,z -Z)}{x_0-x_0^2X_0} - i \frac{\tilde{f}(x_0,x_1,y,z)}{x_0}  
    \right) \\  
  &=&\exp \left( i f(x_1,y)X_0 - ix_1  \frac{\partial}{\partial x_1}f(x_1,y) X_1
   - i \sum_i \frac{\partial}{\partial y_i}f(x_1,y) Y_i  + O(x_0)\right)
 \end{eqnarray*}
Thus, for $\tau = f(x_1,y)$,$\tilde{\sigma} = \frac{\partial}{\partial x_1} f(x_1,y)$ ,
$\xi_i= \frac{\partial}{\partial y_i} f(x_1,y)$, the symbol of $\tilde{N}(\tau,\sigma,\xi)$ is given by $p(0,x_1,y,z,-\tau,x_1 \tilde{\sigma},\xi,\eta)$, in particular
$\tilde{N_0}$ is well-defined and does not depend on the choice of $f$ or $\tilde{f}$ and depends smoothly in $(\tau,\tilde{\eta})$.

Note that if $(\tilde{\sigma}, \xi_i) = (\tilde{\sigma} dx_1 , \sum \xi_i dy_i)$ is a coordinate for $T^*Y$, 
then $(x_1 \tilde{\sigma} dx_1/x_1 , \sum \xi_i dy_i)$ is a coordinate for ${}^b{T^*Y}$. So if $\beta_b:TY \to {}^bTY$ is a blow-down map, the above symbol expression implies that there is a unique map $\hat{N_0} :(\underline{\R} \oplus T^*Y )_y
\to \Psi^m(\phi^{-1}(p);E,F)$ such that $\tilde{N_0} \circ \beta_b = \hat{N_0}$,
 i.e. $\hat{N_0} (\tau,\sigma,\xi) = \tilde{N}(\tau,\sigma/x_1,\xi)$, and its symbol is given by  $p(0,x_1,y,z,-\tau, \sigma,\xi,\eta)$.

By the fourier transform, it turns out that $\hat{N_0}$ defines a suspended pseudodifferential operator (see \cite{MM}). $N_0 = N_0(P) \in \Psi^m_{\mathrm{sus}(\pbu N Y)} (\partial X)$ , $^\Phi N Y \simeq 
\underline{\mathbb{R}}\oplus {}^bT^*Y$. And we say $N_0(P)$ is a normal operator of $P$ on $\pz X$.

We obtain the exact sequence
$$
0 \to x_0 \Psi^{m} \pbs (X) \to \Psi^{m} \pbs (X) \xrightarrow{N_0} \Psi^m_{\mathrm{sus}(\pbu N Y)} (\pz X) \to 0.
$$

We consider a normal operator of $P$ at $\po X$.
For  $P \in \Psi^m \pbs (X;E,F)$ , $\lambda \in \C$ define
$$
\hat{N_1}(P)(\lambda)=[x_1^{-i\lambda} P x_1^{i \lambda}]|_{\po X} \in \Psi^m_\Phi(\po X;E,F).
$$
Obviously, $\hat{N_1}: \C \to \Psi^m_\Phi(\po X;E,F)$ is an entire holomorphic function.

Let $\widetilde{\po X} \simeq \po X \times [0,\infty]$ is the compactification of the positive normal bundle of $\po X \subset X$, then
$\widetilde{\po X}$ obviously has a structure of a manifold with fibered boundary.
Define 
$$\Psi^m_{\Phi,b ,\mathrm{inv}} (\widetilde{\po X}):= \{B \in \Psi^m_{\Phi,b } (\widetilde{\po X}) | \text{ $B$ is equivariant with respect to the 
$(0,\infty)$ action on $\widetilde{\po X}$ } \}.$$
Note that the first front faces $F_1 \subset X^2 \pbs$ and $F_1 \subset \widetilde{\po X}^2 \pbs$ are canonically diffeomorphic, so by Mellin transformation, it turns out that $\hat{N_1}(P)$ defines $N_1(P) \in \Psi^m_{\Phi,b ,\mathrm{inv}} (\widetilde{\po X}) $ and the following sequence is exact (see \cite{M_b} for more detail).
$$
0 \to x_1 \Psi^{m} \pbs(X) \to \Psi^{m}_b(X) \xrightarrow{N_1} \Psi^m_{\Phi,\mathrm{inv}} (\widetilde{\po X}) \to 0
$$

In a local coordinate, define $t:=\log(1-x_0 X_1)/x_0$, then t is smooth up to $x_0=0$, and $X_1=(1-e^{tx_0})/x_0$ is also smooth up to  $x_0=0$.
And $\frac{x_1'^{i \lambda}}{x_1^{i \lambda}}= e^{i x_0 \lambda t}$

Thus by changing coordinate form $X_1$ to $t$ , 
$(x_0,x_1,X_0,t,y,Y,z,Z)$ also gives a coordinate.
Define the symbol $\tilde{p}$ of $P$ with respect to this coordinate by following.
 $$
P(x_0,x_1,X_0,t,y,Y,z,Z) = \left( \int e^{i \sigma_0 X_0} e^{i \sigma_1 t} e^{i \eta Y} e^{i \zeta Z} \tilde{p}(x_0,x_1,y,z,\sigma_0,\sigma_1,\eta,\zeta) d\sigma_0 d\sigma_1
d\eta d\zeta \right) |dX_0 dt dY dZ |
 $$
Then the symbol of $\hat{N_1}(P)(\lambda)$ is $\tilde{p}(x_0,0,y,z,\sigma_0,- x_0\lambda,\eta,\zeta)$.


The normal operators $N_0,N_1$ can be thought as the restriction to  $\po X, \pz X$,
and the symbol $\sigma$ can be thought as the restriction to the boundary $S(\pbu T^* X)$ of $\pbu T^* X$ at infinity.
We want to consider further restriction to the intersection of these two.

As in the case of $\sigma$, we can define symbol maps 
$$\sigma_0 :\sus{m} \to S^m(\pbu T^*X|_{\pz Z};\Hom(E,F) )$$
 and 
$$\sigma_1 : \inv{m} \to S^m(\pbu T^*X| _{\po X};\Hom(E,F)).$$

To consider the restriction to $\angle X$, as $\widetilde{\po X}$ is also a manifold with fibered boundary,
we can define a normal operator on $\pz (\widetilde{\po X}) = \widetilde{\angle X}$,
$N_{0,1}:\inv{m} \to \Psi^m_{\mathrm{sus}(\pbu N \widetilde{\partial Y})}(\widetilde {\angle X} ;E,F)$.

Where
$\widetilde{\partial Y} \simeq \partial Y \times [0,\infty]$ ,$\widetilde {\angle X} \simeq \angle X \times [0,\infty]$
Note that for $Q \in \inv{m}$, $N_{0,1} (Q)$ is also equivariant to the action of $(0,\infty)$, and
$(0, \infty)$ acts on $\widetilde{\angle X}$ or $\widetilde{\partial Y}$ by multiplication.
So the restriction of $N_{0,1}$ to the any fiber of $\phi$ gives the same value in $\Psi^m_{\mathrm{sus}(\pbu N {\partial Y})}(\angle X;E,F)$,
where $\pbu N {\partial Y}:= \pbu N Y|_{\partial Y} \simeq \underline{\R} \oplus \underline{\R} \oplus T^*\partial Y$.

Thus we can define 
$$N_{1,0}:\inv{m} \to \Psi^m_{\mathrm{sus}(\pbu N {\partial Y})}(\angle X;E,F).$$

On the other hand we can define
$$N_{0,1}: \sus{m} \to \Psi^m_{\mathrm{sus}(\pbu N {\partial Y})}(\angle X;E,F)$$ by restriction.

We can also define the symbol map.

$$\sigma_{0,1} :\Psi^m_{\mathrm{sus}(\pbu N {\partial Y})}(\angle X;E,F) \to S^m(\pbu TX|_{\angle X} ;\Hom(E,F)) $$

Define following maps by restrictions of symbols.
$$S^m(\pbu T^*X; \Hom(E,F)) \to  S^m(\pbu T^*X|_{\pz X};\Hom(E,F))$$
$$S^m(\pbu T^*X; \Hom(E,F) ) \to  S^m(\pbu T^*X|_{\po X};\Hom(E,F))$$
$$S^m(\pbu T^*X|_{\po X};\Hom(E,F) ) \to  S^m(\pbu T^*X|_{\angle X};\Hom(E,F))$$
$$S^m(\pbu T^*X|_{\pz X};\Hom(E,F) ) \to  S^m(\pbu T^*X|_{\angle X};\Hom(E,F))$$
 
In summary, we defined 12 maps $\sigma, \sigma_0, \sigma_1,\sigma_{0,1} , N_0,N_1 , N_{0,1} N_{1,0}$ and four restriction maps. It is obvious from the 
definition that any pair of compositions which is defined on same spaces coincide, e.g. $N_{0,1} N_{0}= N_{1,0}N_{1}$ or 
$\sigma |_{\po X} = \sigma_0{N_0}$.
And exact sequences exist for all of these 12 maps as shown in the case $\sigma, N_0$ and $N_1$, but we will omit here.

Finally we can consider the joint symbol $J^m$ which is defined as follows.

\begin{equation*}
\begin{split}
J^m(X;E,F) :=\{ (s,n_0,n_1) \in S^m(\pbu T^*X;\Hom(E,F)) \oplus \sus{m} \oplus \inv{m} \mid \\ 
s|_{\po X} = \sigma_0(n_0),
 s|_{\pz X} = \sigma_1 (n_1), 
 N_{0,1}(n_0)= N_{1,0}(n_1) \}
\end{split}
\end{equation*}

Then the following sequence is exact by the diagram chasing.

\begin{eqnarray}
0 \to x_0x_1 \pseu{m-1} \to \pseu{m} \xrightarrow{\sigma \oplus N_0 \oplus N_1} J^m(X;E,F) \to 0
\label{sequence}
\end{eqnarray}

\section{the Fredholm condition and the relative index theorem}\label{norm}
In this section we prove the
 Fredholm condition and the relative index theorem for fibered cusp $b$-pseudodifferential operators.

Let $X$ be a manifold with fibered boundary. In this section fix a Riemannian metric $g$ on $\pbu TX$, then $g$ can be considered as a
singular metric on $TX$ and $\interior X$ is a Riemannian manifold with respect to that metric. And
we also assume that every complex vector bundle $E$ on $X$ has a hermitian metric $h$.
Then, we can define the $L^2$ space.
$$
L^2 \pbs (X;E) := \{ u \mid \text{$u$ is a metrizable section of $E$ and $\int ||u||^2 dg  < \infty$}  \}
$$
For $u,v \in L^2 \pbs (X;E)$, we can define the inner product by $(u,v):= \int h(u,v) dg$. and $L^2 \pbs (X;E)$ is a Hilbert space
with respect to this inner product.

Let $P \in \Psi^m \pbs (X;E,F)$ then its formal adjoint $P^* \in \Psi^m \pbs (X;F,E) $ can be defined.

To prove $P$ is bounded on $L^2$, we need some technical preparations.
Consider the normal operator on $\po X$, $N_1: \Psi^m \pbs (X;E,F) \to \Psi^m _{\Phi,b,\mathrm{inv}} (\widetilde{\po X};E,F)$, then as illustrated in 
\cite {MN}, we can define a section of $N_1$ as following. Fix a diffeomorphism $\widetilde{\po X} \simeq \po X \times [0, \infty]$ and a collar neighbourhood 
$\po X \times [0, \infty] \hookrightarrow X$ and take a smooth function $\psi$ on $X$ which is supported in a small neighbourhood
of $\po X$ and identically equal to 1 around the neighbourhood of $\po X$. Then the multiplication by $\psi$ can be regarded as a operator $M_\psi : C^\infty (\widetilde{\po X}) \to C^\infty (X)$.
 Define $S: \Psi^m _{\Phi,b,\mathrm{inv}} (\widetilde{\po X} ;E,F) \to \Psi^m \pbs (X;E,F)$ 
 by $S(B)= M_\psi B M_\psi^*$. By definition $S$ is smooth with respect to the Fr\'echt space topology.

As in the case of $b$-calculus, for $B \in \Psi^0 _{\Phi,b,\mathrm{inv}} (\widetilde{\po X} ;E,F)$ and $u \in \dot{C}^\infty 
(\widetilde{\po X};E)$,
the action of $B$ is characterized as follows.
$$\widehat{Bu}(\lambda) = \hat{B}(\lambda) \hat{u}(\lambda)$$
Where $\hat{u}$ and $\widehat{Bu}$ are Mellin transform of $u$ or $\widehat{Bu}$, and $\hat{B}(\lambda) = \hat{N_1}(B)(\lambda)$.

 By the coordinate representation given in section \ref{normal}, for $\lambda \in \R$
 $\hat{B}(\lambda)$ is bounded with respect to the Fr\'chet space topology on $\Psi^0_\Phi(\po X;E,F)$.
 And by \cite{MM}, the embedding $\Psi^0_\Phi(\po X;E,F) \to \mathcal{L}(L^2_\Phi (\po X;E) , L^2_\Phi (\po X;F))$ is bounded,
 so $||\hat{B}(\lambda)||$ is bounded.
 
 Thus $B$ is  also bounded because Mellin transform is an isomorphism on a $L^2$ space.

\begin{proposition}
For $P \in \Psi ^0 \pbs (X;E,F)$, $P$ is bonded as an operator $L^2\pbs(E) \to L^2\pbs(F)$.
And the inclusion $\pseu{0} \to \mathcal{L} (L^2 \pbs(E), L^2 \pbs (F))$ is bounded.
\end{proposition}
\begin{proof}

First, we show the proposition holds for $P\in x_0^Nx_1^N\Psi ^{-N} \pbs (X;E,F)$ and  $N>0$ is sufficiently large.
In this case  the kernel of $P$ blows down and can be written as a continuous kernel on $X^2$ and the boundedness is obvious.

Secondly, we show the proposition holds for $P\in x_0^\epsilon x_1^\epsilon \Psi ^{-\epsilon} 
\pbs (X;E,F)$ for any $\epsilon>0$.
Because $||P||^2 = ||P^* P|| $ and $P^* P \in x_0^ {2\epsilon} x_1^ {2\epsilon}\Psi ^{-2\epsilon}$ , using this discussion recursively, 
the boundedness follows by first step.

Lastly, we consider the general case $P \in \Psi ^0 \pbs (X;E,F)$.
As noted above $N_1(P)$ is bounded with respect to $L^2$ norms,
so $S(N_1(P))= M_\psi S(N_1(P)) M_\psi^* $ is also bounded with respect to $L^2$ norms, because $M_\psi$ is obviously bounded with respect to
$L^2$ norms. By replacing $P$ by $P-S(N_1(P))$, we can assume that $N_1(P)=0$ i.e. $P \in x_1 \Psi ^0 \pbs (X;E,F)$

Take sufficiently large $C>0$ so that  $C-N_0(P^* P)$ and $C- \sigma(P^* P)$ are positive.
Then, we can find a formally self adjoint operator $A \in \Psi ^0 \pbs (X;E,E)$ such that 
$N_0(A) = \sqrt{C- N_1(P^*P)}$ ,$N_1(A)= \sqrt{C}$ and $\sigma(A) = \sqrt{C-\sigma(P^* P)}$. 
Note that $\sqrt{C- N_1(P^*P)}$ can be defined because the calculus of the suspended pseudodifferential operator is closed under holomorphic functional calculus.
Set $B:= C-P^*P -A^2$ , then $N_0(B)=N_1(B)=\sigma(B)=0$ so $B \in x_0 x_1 \Psi ^{-1} \pbs (X;E,E) $.
Thus $B$ is $L^2$ bounded by second step.
$||Pu||^2 = (P^*P u , u) = -(Bu,u)+C||u||^2 - ||Au||^2 \leq (||B||+||C||) \cdot ||u||^2$

\end{proof}

As $\Psi^0_c(\mathcal{G};E,F) \subset \pseu{0}$ is obviously dense with respect to the $L^2$ norm, following propositions can be reduced to
the general theory of pseudodifferential operator of a groupoid \cite{NWX,LMN1}.

By similar arguments we can also prove the following proposition.
\begin{proposition}
 Any one of the 12 maps in the last part of the section \ref{normal}, including $\sigma$, $N_0$ and $N_1$, is bounded with respect to the $L^2$ norm.
\end{proposition}

\begin{theorem}\label{fredholm}
$P \in \Psi^0 \pbs (X;E,F)$ is Fredholm if and only if $\sigma(P)$ and $N_0(P)$ are invertible 
and $\hat{N_1}(P) (t) $ is invertible for all $t \in \R$. 
\end{theorem}
\begin{proof}
As described in \cite{LMN1}, by diagram chasing, the $L^2$ completion of the exact sequence (\ref{sequence}) is also exact,

$$
0 \to \mathcal{K}(X;E,F) \to \overline{\Psi}^0_{\Phi,b}(X;E,F) \xrightarrow{j} \overline{J}^0(X;E,F) \to 0
$$
where $\mathcal{K}$ is the space of compact operators.

Thus $P$ is Fredholm if and only if $j(P)$ is invertible if and only if $\sigma(P),N_0(P),N_1(P)$ are invertible in its $L^2$ closure.

 Because $S^0(\pbu T^*X;E,F)$ and $\sus{0}$ are closed under holomorphic functional calculus, 
 $\sigma(P)$ and $N_0(P)$ are invertible if and only if they are invertible in its completion.
 
 For $N_1(P)\in \inv{0}$ , there is an injective $*$-homomorphism defined by the Mellin transform.
 
 $$
 B \in \inv{0} \to \hat{B}|_\R \in C_b(\R , \Psi^0_\Phi(\po X; E,F))
 $$
 where $C_b(\R , \Psi^0_\Phi(\po X; E,F))$ is a space of bounded continuous function from 
 $\R $ to $\Psi^m_\Phi(\po X; E,F)$. Obviously, the completion of $C_b(\R, \Psi^0_\Phi(\po X; E,F))$ is
 $C_b(\R , \overline{\Psi}^0_\Phi(\po X; E,F))$ and the above map extends to an injective *-homomorphism.
 
 $$
 \overline{\Psi}^0_{\Phi,b,\mathrm{inv}}(\widetilde{\po X}; E,F) \to C_b(\R , \overline{\Psi}^0_\Phi(\po X; E,F)).
 $$
 
 Thus $N_1(P)$ is invertible in its completion if and only if its image in $ C_b(\R ,\overline{\Psi}^0_\Phi(\po X; E,F))$ is invertible,
and the claim follows.
\end{proof}

Finally, we move on to the proof of the relative index theorem.
\begin{lemma}
Let $P \in \Psi ^0\pbs (X;E,F)$  and suppose that $\sigma(P)$ and $N_0(P)$ are invertible.
Then there is a parametrix $Q \in \Psi ^{-m}\pbs (X;F,E)$ such that 
$PQ-\id \in x_0^{\infty}\Psi ^{-\infty}\pbs (X;F,F)$ , $QP-\id \in x_0^{\infty}
\Psi ^{-\infty}\pbs (X;E,E)$.
\end{lemma}
\begin{proof}

 We construct the right parametrix $Q$ inductively. Take $Q_0 \in \Psi^0 \pbs (X;F,E)$ 
so that $\sigma_{0}(Q_0)= \sigma_{0}(P)^{-1}$, $N_0(Q_0)= \sigma(P)^{-1}$.
 Then $PQ_0-\id \in x_0 \Psi^{-1} \pbs (X;F,F)$.
 
 Set $R_0:= (PQ_0-\id)/x_0 \in \Psi^{-1} \pbs (X;E,E)$ . And take $Q_1 \in \Psi^{-1} \pbs (X;F,E)$ such that
 $\sigma_{-1}(Q_1)= - \sigma_{0}(P)^{-1} \sigma_{-1} (R_0)$, $N_0(Q_1)= - N_0 (P)^{-1} N_0(R_0)$. 
 By definition of $Q_1$, $P(Q_0+x_0Q_1)- \id = x_0(R_0 + PQ_1) \in x_0^2\Psi^{-2} \pbs (X;F,F)$.

Suppose we constructed $Q_1, \dots Q_n$ such that $Q_m \in \Psi^{-m} \pbs (X;F,E)$ and 
$P (\sum_0^n x_0^m  Q_m) -\id \in x_0^{n+1}\Psi^{-n-1} \pbs (X;F,F)$.
Set $R_n:= (P (\sum_0^n x_0^m Q_m) -\id )/x_0^{n+1}$. And take $Q_{n+1} \in \Psi^{-n-1} \pbs (X;F,E)$
such that $\sigma_{-n-1}(Q_1)= - \sigma_{0}(P)^{-1} \sigma_{-n-1} (R_n)$, 
$N_0(Q_1)= - N_0 (P)^{-1} N_0(R_n)$. 
Then as above, $P (\sum_0^{n+1} x_0^m  Q_m) -\id \in x_0^{n+2}\Psi^{-n-1} \pbs (X;F,F)$.

Finally define $Q:= \sum_0^{\infty} Q_m$ by an asymptotic sum. Then  $Q \in \Psi^0 \pbs (X;F,E)$
and $PQ- \id \in \underset{m} {\cap} x_0^m  \Psi^{-m} \pbs (X;E,E) =x_0^\infty \Psi^{-\infty} (X;F,F)$.
As we can construct left parametrix similarly, $Q$ is actually right and left parametrix.

\end{proof}

For the above parametrix $Q$, define $S =\id - PQ \in x_0^{\infty}\Psi ^{-\infty}\pbs (X;F,F)$. 
Note that $x_0^{\infty}\Psi ^{-\infty}\pbs (X;F,F)= x_0^{\infty}\Psi ^{-\infty}_b (X;F,F)$, 
because the kernel of any element of $x_0^{\infty}\Psi ^{-\infty}\pbs (X;F,F)$ vanishes on $FF_0$ in infinite order
and blows down to the kernel on $X^2_b$.

We can see
$\hat{N_1}(S)(\lambda)=\id - \hat{N_1}(P)(\lambda) \hat{N_1}(Q)(\lambda) $ rapidly decreases as $|\mathrm{Re}{\lambda}| \to \infty $ by the theory of $b$ - calculus. Thus as in \cite{M_b}, we can prove the following lemma.
\begin{lemma}\label{mero}
Let $P \in \Psi ^0\pbs (X;E,F)$  and suppose that $\sigma(P)$ and $N_0(P)$ are invertible.
$\hat{N_1}(P)(\lambda)^{-1}$ is a meromorphic map from $\C$ to $\Psi ^{0}\pbs (X;F,E)$ 
such that for any $N>0$, there exists $C>0$ such that $\hat{N_1}(P)^{-1}(\lambda)$ exists and bounded on
$\{\lambda \in \C \mid  \text{$|\mathrm{Re} \lambda| > C$ and $|\mathrm{Im} \lambda| <N$ } \}$.

In particular,the number of poles in the strip $\{\lambda \in C \mid |\mathrm{Im} \lambda| <N \}$ is finite.
\end{lemma}

Let $P \in \Psi ^0\pbs (X;E,F)$, and suppose that $\sigma(P)$ and $N_0(P)$ are invertible.
Obviously, $\sigma(x_1^\alpha P x_1^{-\alpha}) = \sigma(P)$. And because $x_1$ is constant on each fiber of $\phi$ , $N_0(P)$ commutes with $x_1^\alpha$ and $N_0 (x_1^\alpha P x_1^{-\alpha}) = N_0(P)$.

 For $\beta \in \R$, $\hat{N_1}(x_1^\beta P x_1^{-\beta})(\lambda) = \hat{N_1} (\lambda+i \beta)$. By theorem \ref{fredholm} $x_1^\beta P x_1^{-\beta}$ is Fredholm if and only if 
$\beta \notin - \mathrm{Im} \mathrm{Spec} (\hat{N_1}(P))$. Where $\mathrm{Spec} (\hat{N_1}(P)) := \{ \lambda \in \C \mid \text{ $\hat{N_1}(P)(\lambda)$ is not invertible} \}$ which is discrete by lemma \ref{mero}.

Thus, exactly as in \cite{M_b}, we can prove the relative index theorem.

\begin{theorem}\label{relative}
Let $P \in \Psi^0 \pbs (X;E,F)$ and suppose that $\sigma(P)$ and $N_0(P)$ are invertible, $\beta_i \notin -\mathrm{Im} \mathrm{Spec}_b(P) \ (i=1,2)$ $\beta_2>\beta_1$ .Then,
$$
\ind (x_1^{\beta_1} P x_1^{-\beta_1}) - \ind (x_1^{\beta_2} P x_1^{-\beta_2}) =  \frac{1}{2 \pi i} \mathrm{tr} \oint
 \hat{N_1}(P)^{-1} (\lambda) \frac{\partial  \hat{N_1}(P)} { \partial \lambda} (\lambda) d \lambda,
$$
where $\ind$ is the index of Fredholm operator, $\mathrm{tr}$ is the trace, and the
integral path is chosen so that its interior contains all poles of $\hat{N}(P)^{-1}(\lambda)$ such that $\beta_1 < -\mathrm{Im}(\lambda) < \beta_2$.
\end{theorem}

For $P \in \Psi^0 \pbs (X;E,F)$ and $I:=[\delta,\gamma] \subset \R$ , $\delta< \gamma$ be a closed interval.
Define a norm by $||P||_I:= \sup_{\alpha \in I} ||x_1^\alpha P x_1^{-\alpha}||$. And $\overline{\Psi^0 \pbs}^I (X;E,F)$ be a completion with respect to that norm.
Then $\hat{N_1}$ extends.

$$
\hat{N_1} : \overline{\Psi^0 \pbs}^I (X;E,F) \to \mathrm{Hol}_b(\R \times i I, \overline{\Psi}^0_{\Phi} (\po X;E,F)),
$$
where $\R \times i I = \{ \lambda \in C \mid \delta \leq \mathrm{Im}(\lambda) \leq \gamma \}$ and 
$\mathrm{Hol}_b$ is a space of bounded continuous function which is holomorphic in the interior.

To see $\hat{N_1}$ extends to the completion, let $P \in \overline{\Psi^0 \pbs}^I (X;E,F)$ and 
 $P_n \in \Psi^0 \pbs (X;E,F)$ such that $||P-P_n||_I \to 0$, then $\hat{N_1}(P_n)|_{\R \times i I}$ is a Cauchy sequence by the definition of
the norm, and uniformly converges to some $\hat{N_1}(P)$. Because uniform limit of holomorphic function is holomorphic, $\hat{N_1}(P)$ is holomorphic in the interior.

Note that $\sigma$  and $N_0$ also extends because $\sigma(x_1^\alpha P x_1^{-\alpha}) = \sigma(P)$ and 
$N_0 (x_1^\alpha P x_1^{-\alpha}) = N_0(P)$.

And theorem \ref{relative} can be extended to the completion $P \in \overline{\Psi^0 \pbs}^I (X;E,F)$, because both hand side of the identity are
continuous with respect to $|| \cdot ||_I$ and is an integer.

\section{Application to $\mathbb{Z}$/k-manifolds} \label{Zk}
In this section we fix an isomorphism $TX=\pbu T^*X$ for simplicity.
Suppose $X$ is a $\Z/k$ manifold with boundary , i.e.
$X$ is a manifold with corner and $\partial X = \partial _0 X \cup \partial _1 X \  , \  \angle X = \partial _0 X \cap \partial _1 X$ and the diffeomorhpism $\partial _1 X \simeq k Z$ is given,
where $Z$ is a manifold with boundary and $k Z$ is a disjoint union of $k$  copies of $Z$. 
For $\phi=\id : \pz X \to \pz X$ , we regard $X$ as a manifold with fibered boundary.
And we write $\Psi^0_{sc,b} (X;E,F)= \pseu{0}$ in this case.

A vector bundle $E$ over $X$ is called $\Z/k$-vector bundle if $E|_{\po X} = k E_Z$ for some vector bundle $E_Z \to Z$.
Fix a $\Z/k$-vector bundle structure on $TX$.
Define $\overline{X}$ be a quotient of $X$  obtained by identifying $k$ copies of $Z$ in $X$.
Then $TX \to X$ descends to a vector bundle $\overline{TX} \to \overline{X}$.

Let $E,F$ are $\Z/k$- vector bundle over $X$, and $P \in \Psi^0 \pbs (X;E,F)$.
Define
 $$
\scat := \{ P \in \Psi^0_{sc,b} (X;E,F) \mid \text{ $N_1(P) = k Q$ for some $Q \in \Psi^0_{sc,b} (\widetilde{\po Z} ;E,F)$ } \}
$$

Where $k Q = Q \oplus Q \dots \oplus Q \in \Psi^0_{sc,b} (\widetilde{\po X} ;E,F)$ is defined by using isomorphism $\po X \simeq k Z$.

$\hat{N_0 }(P)$ is a bundle homomorphism over $\pz X$ , $\hat{N_0 }(P): \underline{\R} \oplus T(\po X) \to \underset{y \in \pz X}{\cup} \Psi^0(\phi^{-1} (y); E,F)$.
And note that $\phi^{-1} (y)$ is one point set in this case. So $\Psi^0(\phi^{-1} (y); E,F) \simeq \Hom (E,F)$.
Under this identification, by the compatibility of $N_0$ and $\sigma$, $\sigma(P)(\xi) = \lim_{t \to \infty}
 \hat{N_0}(t\xi)$ for $\xi \in S(TX|_{\pz X})$.
 
 By above observations, there is a map 
\begin{equation*}
\begin{split}
s: P \in \{ P \in  \sc \mid 
 \text{ $\sigma(P)$ and $N_0(P)$ are invertible }\}
\mapsto [E,\sigma(P) \cup N_0(P),F] \\ \in K(D(\overline{TX}),S(\overline{TX}) \cup  D(\overline{TX})|_{\overline{\pz X}}).
\end{split}
\end{equation*}

Where we regard $\sigma (P)\cup N_0 (P)$ as a bundle isomorphism between $E$ and $F$ over $S(\overline{TX}) \cup  D(\overline{TX})|_{\overline{\pz X}}$.

Let $P \in  \sc$ and suppose that $\sigma(P)$ and $N_0(P)$ are invertible.
Then the right hand side of the identity in theorem \ref{relative} is always a multiple of $k$, so $\ind(x_1^{\beta} P x_1^{-\beta}) \mod k \in \Z/k$ does not
depends on $\beta \notin - \mathrm{Im} \mathrm{Spec}(\hat{N_1}(P) (\lambda))$. And more strongly following lemma holds.

\begin{lemma}\label{homotopy}
Let $P,Q \in \sc$ and suppose that $\sigma(P), \sigma (Q)$ and $N_0(P),N_1(P)$ are invertible.
If $(\sigma(P),N_0(P))$ and $(\sigma(Q),N_0(Q))$ are homotopic in the space of invertible joint symbols, then 
$\ind(x_1^{\beta} P x_1^{-\beta})  \equiv \ind(x_1^{\beta} Q x_1^{-\beta}) \mod k$ for $\beta \notin 
- \mathrm{Im} \mathrm{Spec}(\hat{N_1}(P) (\lambda)) \cup - \mathrm{Im} \mathrm{Spec}(\hat{N_1}(Q) (\lambda))$ 
\end{lemma}

\begin{proof} 
Let $(s_t ,n_t) $, $0 \leq t \leq 1$ be a homotopy such that $(s_0,n_0) =   (\sigma(P),N_0(P))$ , $(s_1,n_1) = (\sigma(Q),N_0(Q))$.
Take any lift $R_t$ of $(s_t,n_t)$ , i.e. $(\sigma(R_t),N_0(R_t)) = (s_t,n_t)$.
Combining the homotopies $(1-t)P+ R_0$ and $tQ + (1-t)R_1$, we can get a homotopy $S_t$ such that $S_0=P$, $S_1=Q$ and $(\sigma(S_t),N_0(S_t))$ is invertible
for all $0 \leq t \leq 1$.
Partition the interval sufficiently small $[t_0,t_1], \dots, [t_{m-1},t_m]$ ,$0= t_0 < t_1 < \dots < t_{m-1}< t_{m}=1$ so that we can choose
$\beta_0 ,\dots ,\beta_{m-1}$ such that $x^{\beta_i} S_t x^{-\beta_i} $ is Fredholm on $[t_i,t_{i+1}]$.

Then $\ind(x^{\beta_i} S_t x^{-\beta_i})$ constant on $[t_i,t_{i+1}]$ and 
$\ind(x^{\beta_i} S_{t_i} x^{-\beta_i}) \equiv \ind(x^{\beta_{i-1}} S_{t_i} x^{-\beta_{i-1}}) \mod k $.
So $\ind(x^{\beta_0} P x^{-\beta_0}) \equiv \ind(x^{\beta_{m-1}} Q x^{-\beta_{m-1}}) \mod k$ and the claim is proved.
\end{proof}

As in \cite{FM} or \cite{W}, we can define a topological index map.
$$
\tind: K(D(\overline{TX}),S(\overline{TX}) \cup  D(\overline{TX})|_{\overline{\pz X}}) \to \Z/k
$$
And the index theorem can be proved.

\begin{theorem}
Let $P \in  \sc$ and suppose that $\sigma(P)$ and $N_0(P)$ are invertible, then
$ \ind(x_1^{\beta}Px_1^{-\beta}) \mod k = \tind(s(P)) \in \Z/k$ ,  $\beta \notin - \mathrm{Im} \mathrm{Spec}(\hat{N_1} (\lambda))$.
\end{theorem}
\begin{proof}
We will demonstrate two different ways to prove the theorem.

The first method is to reduce the case when $\pz X$ is empty as in \cite{M_g}. 
Embed $X$ into $Y$ ,where $Y$ is $\Z/k$- manifold such that $\pz Y = \phi$, e.g. we can take $Y$ as a double of $X$.
Choose $G$ so that $F \oplus G \simeq  \underline{\C}^n$
 , by replacing $P$ by $P \oplus \id_G$, we can assume $F=\underline{\C}^n$ is a trivial bundle.
Because $D(TX|_{\pz X})$ is homotopy equivalent to $\pz X$, by replacing $P$ to homotpic element,
we can assume that $\hat{N_1}(P) : TX|_{\pz X} \to \Hom(E,F)$  is constant on each fiber and given by some bundle isomorphism
  $\theta: E \simeq F={\C}^n$.
  
Using $\theta$, we can extend $E$ onto $Y$ in obvious way.
Take a cut-off function $\phi$ such that $\phi \equiv 1  $ near $\pz X$.
If we choose $\phi$ so that its support is sufficiently small,
 $Q:= \theta \phi + P(1-\phi)$ is homotopic to $P$.
$Q$ can be extended to a b-pseudodifferential operator $\tilde{Q}$ on $Y$  And by construction, under the excision map $K(D(\overline{TX}),S(\overline{TX}) \cup  D(\overline{TX})|_{\overline{\pz X}}) \to K(D(\overline{TY}),S(\overline{TY}))$, $\sigma(Q)$ is mapped 
to $\sigma(\tilde{Q})$ thus, $\tind(s(Q)) = \tind(s(\tilde{Q}))$.
Obviously, $\ind(x_1^{\beta}Qx_1^{-\beta})= \ind(x_1^{\beta} \tilde{Q} x_1^{-\beta})$ for all $\beta \notin - \mathrm{Im} \mathrm{Spec}(\hat{Q} (\lambda))$.

Because $\pz Y$ is empty, by \cite{FM}, $\tind(s(\tilde{Q})) = \ind(x_1^{\beta} \tilde{Q} x_1^{-\beta}) \mod k$ and the theorem is proved.

For the second method, we only give a outline.

We define the analytic index $\aind:K(D(\overline{TX}),S(\overline{TX}) \cup  D(\overline{TX})|_{\overline{\pz X}}) \to \Z/k$
by $\aind(s(P)) = \ind(x^\beta P x^{-\beta}) \mod k$ , $\beta \notin - \mathrm{Im} \mathrm{Spec}(\hat{P} (\lambda)) $.
Then it is well-defined.
And we can prove that $\aind$ satisfies the axioms as in \cite{AS}, \cite{FM} or \cite{W}.

For the part of the axiom about multiplication, we need to pay some attention. Let $W$ be a closed manifold, and $P \in \pseu{m}$ , $m>0$ .
Then in general, as in \cite{AS}, $P \boxtimes \id_W \notin \Psi^m \pbs (X\times W;E,F ) $
but it is contained in the completion $ \overline{\Psi^m \pbs}^I (X\times Z;E,F ) $ defined in section \ref{norm} .

\end{proof}

\end{document}